\documentclass[11pt]{amsart}
\usepackage{amsfonts,amssymb,graphicx}
\usepackage{fourier}
\usepackage{geometry}
\usepackage{amssymb}
\usepackage{mathtools}
\usepackage{tikz}
\usetikzlibrary{chains}

\tikzset{node distance=2em, ch/.style={circle,draw,on chain,inner sep=2pt},chj/.style={ch,join},every path/.style={shorten >=4pt,shorten <=4pt},line width=1pt,baseline=-1ex}

\newcommand{\mlabel}[1]{%
  \(#1\)
}

\let\dlabel=\alabel

\newcommand{\dnode}[2][chj]{%
\node[#1,label={below:\dlabel{#2}}] {};
}

\newcommand{\dnodebr}[1]{%
\node[chj,label={below right:\dlabel{#1}}] {};
}

\newcommand{\dydots}{%
\node[chj,draw=none,inner sep=1pt] {\dots};
}

\numberwithin{equation}{section}

\newtheorem{theorem}{Theorem}[section]
\newtheorem{lemma}[theorem]{Lemma}

\newtheorem{corollary}[theorem]{Corollary}

\newtheorem{conjecture}[theorem]{Conjecture}

\theoremstyle{definition}

\theoremstyle{remark}

\newcommand{\CC}{\mathbf{C}}

\newcommand{\HH}{\mathbf{H}}
\newcommand{\RR}{\mathbf{R}}

\newcommand{\ZZ}{\mathbf{Z}}
\newcommand{\cA}{\mathcal{A}}

\newcommand{\la}{\langle}
\newcommand{\ra}{\rangle}

\begin{document}
\title{Zero fibers of quaternionic quotient singularities}
\author{Lien Cartaya}
\author{Stephen Griffeth}
\date{}

\begin{abstract} We propose a generalization of Haiman's conjecture on the diagonal coinvariant rings of real reflection groups to the context of irreducible quaternionic reflection groups (also known as symplectic reflection groups). For a reflection group $W$ acting on a quaternionic vector space $V$, by regarding $V$ as a complex vector space we consider the scheme-theoretic fiber over zero of the quotient map $\pi:V \to V/W$. For $W$ an irreducible reflection group of (quaternionic) rank at least $6$, we show that the ring of functions on this fiber admits a $(g+1)^n$-dimensional quotient arising from an irreducible representation of a symplectic reflection algebra, where $g=2N/n$ with $N$ the number of reflections in $W$ and $n=\mathrm{dim}_\HH(V)$, and we conjecture that this holds in general. We observe that in fact the degree of the zero fiber is precisely $g+1$ for the rank one groups (corresponding to the Kleinian singularities). In an appendix, we give a proof that three variants of the Coxeter number, including $g$, are integers. \end{abstract}

\thanks{We thank Nicol\'as Libedinsky, Arun Ram, Vic Reiner, Christian Stump, and Ulrich Thiel for comments on a preliminary version of this article. The second author acknowledges the financial support of Fondecyt Proyecto Regular 1231355.}
\maketitle

\section{Introduction}

 The main purpose of the present paper is to generalize the conjecture of Haiman (see \cite{Hai1}, \cite{Hai2}, \cite{Hai3}, and \cite{Gor}) on the diagonal coinvariant rings of real reflection groups to the class of quaternionic reflection groups. We believe that this is the broadest class of groups (at least in characteristic $0$) for which the phenomenon observed by Haiman persists. As evidence for this generalized conjecture, we prove that it holds for all irreducible quaternionic reflection groups of rank at least $6$. 

\subsection{Historical background} Work on this circle of ideas began with Haiman's foundational papers, where he proved that the dimension of the diagonal coinvariant ring of the symmetric group $S_n$ is $(n+1)^{n-1}$. Haiman's conjecture is that, for an irreducible real reflection group with Coxeter number $h$ and rank $n$, there exists an $(h+1)^n$-dimensional quotient ring of the diagonal coinvariant ring. Gordon proved Haiman's conjecture in \cite{Gor} by exhibiting a certain non-commutative deformation of a quotient of the diagonal coinvariant ring, and together with the second author extended this result to complex reflection groups in \cite{GoGr}. After the appearance of \cite{GoGr}, there was no further progress on the problem of understanding the diagonal coinvariant rings of general reflection groups for many years. But two recent discoveries suggest that there is much more to be learned.

First, in \cite{Gri4}, one of us observed that the lower bound $(h+1)^n$ of \cite{GoGr} may be further improved to $(g+1)^n$, where for an irreducible complex reflection group of rank $n$ containing $N$ reflections the number $g$ introduced in \cite{Gri4} is 
$$g=2N/n.$$ We recall that the \emph{Coxeter number} $h$ is
$$h=\frac{N+N^*}{n},$$ with $N^*$ the number of reflecting hyperplanes, so $g \geq h$ with equality precisely when every reflection in $W$ has order $2$ (which occurs in particular when $W$ is real). 

The second disovery (in \cite{AjGr}) is the observation that a refinement of the strategies employed in \cite{Gor}, \cite{GoGr}, and \cite{Gri4} (see also \cite{Val} and \cite{Gri1})  explains at least some of the difference between the dimension of the diagonal coinvariant ring and the lower bound $(h+1)^n$. These two discoveries together showed that the representation theory of non-commutative deformations of the quotient map $V \to V/W$ knows far more about the structure of this map than we had previously presumed: firstly, because there always exists an irreducible representation $L$ of a suitable deformation of $V \to V/W$ with $\mathrm{dim}(L)$ very close to the dimension of the diagonal coinvariant ring, and secondly because the difference between the dimensions seems to be accounted for by other irreducible representations.  But the reasons for the extent of this knowledge remain mysterious; we know of no \emph{a priori} reason for these phenomena. Our observation here is then that the natural generality for this mysterious connection is the class of quaternionic reflection groups, for which the relevant non-commutative deformations are certain representations of the corresponding symplectic reflection algebra.

\subsection{Tools} In all the papers since \cite{Gor}, the strategy for obtaining these lower bounds is to find an irreducible representation $L$ of the symplectic reflection algebra (a deformation of the ring $\CC[V] \rtimes W$) with the property that the space of $W$-invariants in $L$ is one-dimensional (actually, beginning with Gordon it has been traditional to look for representations in which the determinant representation of $W$ occurs exactly once, but upon twisting by the inverse of the determinant this amounts to the same thing). The dimension of $L$ then gives the desired lower bound (which turns out, in all the cases in which we can explicitly compute the correct answer, to be either correct or mysteriously close to the truth). This is what we will do here, after formalizing this strategy in \ref{non-commutative fibers}. 

The tactics in each case tend to be quite subtle, due to the difficulty of calculating explicit character formulas for finite-dimensional representations of the relevant deformations. In particular, the principal tool applied in \cite{Gor} was the Heckman-Opdam shift functor; in \cite{GoGr} it was Rouquier's theorem \cite{Rou} on the uniqueness of highest weight covers; \cite{Gri4} used Rouquier's theorem together with the work \cite{BeCh} of Berest and Chalykh on quasi-invariants; and \cite{AjGr} used the representation-valued orthogonal polynomials introduced in \cite{Gri2} together with the results of \cite{Gri1} and \cite{FGM}. In this paper, our main technical inputs are Crawley-Boevey and Holland's study \cite{CBH} of the representation theory of what are nowadays called rank one symplectic reflection algebras. In the end, that part of \cite{CBH} that we use boils down to a detailed understanding of how certain reflection functors relate representations of these algebras at different values of the deformation parameter to one another, and so part of what we do here might be thought of as directly analogous to Gordon's reliance on shift functors. In the remainder of this introduction we give precise definitions and statements of our results.

\subsection{The zero fiber ring} Let $F$ be a field, let $V$ be a finite-dimensional $F$-vector space, and let $W \subseteq \mathrm{GL}(V)$ be a finite group. The ring of functions on the scheme $V/W$ is, by definition, the ring $F[V]^W$ of $W$-invariants in the ring $F[V]$ of polynomial functions on $V$, and the inclusion $F[V]^W \hookrightarrow F[V]$ corresponds to the projection $\pi:V \to V/W$. The most interesting fiber of this projection is the fiber $\pi^{-1}(0)$ over $0$, which we call the \emph{zero-fiber}. It is the subscheme of $V$ corresponding to the quotient $F[\pi^{-1}(0)]=F[V]/I$ of $F[V]$, where $I$ is the ideal of $F[V]$ generated by all $f-f(0)$ for $f \in F[V]^W$. When $W$ is a reflection group acting in $V$, it has been traditional to call this ring the \emph{coinvariant ring} of $W$, which conflicts with the use of the terminology \emph{coinvariant} for the largest quotient object on which the group acts trivially. Therefore we take this opportunity to refer to this ring in the more general context of an arbitrary linear group as the \emph{zero-fiber ring}. This terminology also has the advantage of describing somewhat more vividly what the ring actually \emph{is}.

The \emph{degree} $\delta=\delta(W)$ of the zero-fiber is the dimension of the zero-fiber ring. It is a very interesting and subtle invariant of the linear group $W$. When the base field is of characteristic zero, it is not hard to see that we have the bounds
$$d \leq \delta(W) \leq {n+d-1 \choose d-1} \quad \hbox{where $d=|W|$ and $n=\mathrm{dim}(V)$.}$$ These bounds are moreover sharp in the sense that for given $d$ and $n$ there is a group $W$ of order $d$ acting on a vector space of dimension $n$ that achieves each of them. Namely, the lower bound is achieved precisely when $W$ is a reflection group, and the upper bound is achieved exactly when $W$ is a group of scalar matrices. Thus there is nothing more to say without singling out some class of groups with more structure. Here we propose that a particularly interesting class of groups to study is the class of \emph{quaternionic reflection groups}. See, for instance, \cite{McK3} (finite subgroups of the quaternions) and \cite{GNS} (finite subgroups of $\mathrm{SL}_3(\CC)$) for earlier studies of this problem for certain classes of groups.
 
\subsection{Quaternionic reflection groups} We consider the division rings $\RR \subseteq \CC \subseteq \HH$ of real numbers, complex numbers, and quaternions. Let $D$ be one of $\RR$, $\CC$, or $\HH$. Given a finite-dimensional right $D$-vector space $V$ and a finite subgroup $W \leq \mathrm{GL}(V)$, a \emph{reflection} in $W$ is an element $r \in W$ such that $\mathrm{codim}_D(\mathrm{fix}_V(r))=1$. The group $W$ is a \emph{reflection group} if it is generated by the set $R$ of reflections it contains. If $D=\RR$ or $\CC$ and $W$ is a reflection group acting in the vector space $V$, then $W \otimes 1$ is a reflection group acting in $V \otimes_D \HH$. Thus each real or complex reflection group may be regarded as a quaternionic reflection group via extension of scalars. We will refer to quaternionic reflection groups (including those that arise via extension of scalars from real or complex reflection groups) alternatively as \emph{symplectic reflection groups}; in the next paragraph we briefly explain the reason for this terminology. For now we note that according to Cohen \cite{Coh}, the irreducible quaternionic reflection groups that do not arise from complex groups may be usefully divided into two classes, the complex-primitive and complex-imprimitive groups. There are three infinite families of imprimitive groups, two of which exist only in quaternionic rank two, and a primitive group is one of a finite list of exceptional groups, all of rank at most five. In this paper we are concerned mainly with the infinite family that exists in arbitrarily large dimension of irreducible imprimitive groups, which are certain normal subgroups of wreath products of the symmetric group with finite subgroups of $\HH^\times$.

Let $W$ be a quaternionic reflection group acting on $V$. There exists a $W$-invariant positive-definite Hermitian form $(\cdot,\cdot):V \times V \to \HH$. Now $\HH$ is the (right) $\CC$-vector space with basis $1,j$. The form $(\cdot,\cdot)$ may therefore be written uniquely as
$$(v_1,v_2)=\la v_1,v_2 \ra' + j \la v_1,v_2 \ra$$ for certain $\CC$-valued functions $\la \cdot,\cdot \ra$ and $\la \cdot,\cdot \ra'$. One checks that in fact $\la \cdot,\cdot \ra$ is a $W$-invariant non-degenerate $\CC$-biliinear alternating form on the $\CC$-vector space $V$, and that $\la \cdot,\cdot \ra'$ is a $W$-invariant positive definite Hermitian form on the $\CC$-vector space $V$, which are moreover related by the formula
$$\la v_1,v_2 \ra'=\overline{\la v_1,v_2 j \ra}.$$ Thus each quaternionic reflection group $W$ produces, upon restriction of scalars to $\CC \subseteq \HH$, a symplectic $\CC$-vector space $V$ on which $W$ acts by symplectic transformations, and such that the resulting group is generated by elements with fixed-space of complex codimension $2$ (these are sometimes known as \emph{symplectic reflections}). It is this $\CC$-vector space $V$ together with the quotient map $\pi:V \to V/W$ that we will study.

\subsection{Results and conjectures on the zero-fiber ring of a quaternionic reflection group}

\begin{theorem} \label{main}
Let $W$ be either  an irreducible imprimitive quaternionic reflection group of rank at least $3$ or the quaternionification of an irreducible complex reflection group acting on an $n$-dimensional $\HH$-vector space $V$. Suppose $W$ contains $N$ reflections and let $g=2N/n$. Then there exists a quotient of the zero-fiber ring of  complex dimension $(g+1)^n$. Moreover, if $W$ is a quaternionic group of rank $1$ then the zero fiber ring is of complex dimension exactly $g+1=2|W|-1$. 
\end{theorem} We remark that in case $W$ is the quaternionification of a compex reflection group, the theorem follows from the results of \cite{Gri4}. Therefore, aside from rank one we are reduced to studying the case of imprimitive groups of rank at least $3$, which we do using Cohen's classification \cite{Coh} and a case by case check. The assertion in the abstract follow from Theorem \ref{main} by a second application of the classification. Of course a more satisfying explanation (for instance generalizing what is done in \cite{Gri4}) would be very desirable. We state our hope as a conjecture (which is a quaternionic version of the  conjecture of Haiman \cite{Hai1} on the diagonal coinvariant rings of real reflection groups; in the appendix we prove that the number $g$ appearing here is an integer):
\begin{conjecture}
Let $W$ be an irreducible quaternionic reflection group containing $N$ reflections and acting in the $n$-dimensional quaternionic vector space $V$. Put $g=2N/n$. Then there is a $(g+1)^n$-dimensional quotient ring of the zero-fiber ring of $W$.
\end{conjecture} The results of \cite{EtMo} imply that the quotients we construct here arise from irreducible representations that exist at generic points in a certain hyperplane in the parameter space for the symplectic reflection algebra; the same phenomenon occurs for the representations constructed in \cite{Gri4}. So one might optimistically conjecture that this happens for general quaternionic reflection groups. One might additionally hope that these representations produce an extension of some of the constructions from \cite{ARR} to the case of quaternionic reflection groups.

There is another phenomenon which we would like to understand: for $W$ a complex reflection group, the space of semi-invariants for the determinant representation in the non-commutative fiber constructed in \cite{GoGr} is of dimension equal to the $W$-Catalan number. These numbers were first introduced in \cite{BeRe}, who conjectured part of the result later proved by \cite{GoGr}; surprisingly, this same dimension occurs also for the larger non-commutative fiber from \cite{Gri4}. In these cases, there is an additional internal grading on the representations involved, and both the dimensions and their $q$-analogs factor in a nice way. In \ref{Catalan numbers}, we observe a similar phenomenon for the (ungraded) dimensions of certain spaces of semi-invariants. It would be interesting to investigate if this phenomenon persists for other classes of quaternionic reflection groups, and if conjectures along the lines of those from \cite{Stu} might be formulated. 

\section{Preliminaries}

\subsection{Quaternions} We write $\HH$ for the ring of quaternions, which is the $\RR$-algebra with basis $1,i,j,k$ and multiplication determined by the rules
$$i^2=j^2=k^2=-1=ijk.$$ The \emph{conjugate} of a quaternion is written as $$\overline{a+bi+cj+dk}=a-bi-cj-dk \quad \hbox{for $a,b,c,d \in \RR$,}$$ and this defines an $\RR$-linear anti-automorphism of $\HH$. The \emph{norm} $|z|$ of $z \in \HH$ is then determined by $|z|^2=\overline{z} z \in \RR$. We write 
$\HH^n$ for the right $\HH$-vector space of column vectors of length $n$ with entries in $\HH$, and thus identify $\mathrm{End}_{\HH}(\HH^n)$ with the matrix ring $\mathrm{Mat}_{n \times n}(\HH)$ of $n$ by $n$ matrices with entries in $\HH$, acting on $\HH^n$ by left multiplication. We will write $\mathrm{GL}_n(\HH)$ for the group of invertible elements of this matrix ring, and define the \emph{quaternionic unitary group}
$$\mathrm{U}_n(\HH)=\{g \in \mathrm{GL}_n(\HH) \ | \ \overline{g}^t=g^{-1} \},$$ where for a matrix $g$ we write $\overline{g}$ for the matrix obtained by conjugating all the entries of $g$ and $\overline{g}^t$ for its transpose. In particular, the group $\mathrm{U}_1(\HH) \leq \HH^\times$ is the group of quatenions of norm $1$, and contains each finite subgroup $\Gamma$ of $\HH^\times$. In general, the group $\mathrm{U}_n(\HH)$ consists of the elements of $\mathrm{GL}_n(\HH)$ preserving the positive definite Hermitian form
$$(x,y)=\sum_{p=1}^n \overline{x_p} y_p, \quad \hbox{where $x,y \in \HH^n$.}$$ 

\subsection{Quaternionic reflection groups} Each finite subgroup $W$ of $\mathrm{GL}_n(\HH)$ is conjugate to a finite subgroup of the unitary group $\mathrm{U}_n(\HH)$. Given a finite subgroup $W \leq \mathrm{U}_n(\HH)$, we write 
$$R=\{r \in W \ | \ \mathrm{dim}_\HH(\mathrm{fix}(r))=n-1 \}$$ for the set of \emph{reflections} it contains. The group $W$ is a \emph{quaternionic reflection group} if it is generated by $R$. For instance, each complex reflection group produces a quaternionic reflection group via extension of scalars. Cohen \cite{Coh} has classified the irreducible quaternionic reflection groups. Given an irreducible quaternionic reflection group we define $N=|R|$ and $g=2N/n$, an analog of the Coxeter number of $W$ (for the quaternionification of a complex reflection group with reflections of order greater than two, $g$ is strictly larger than the usual Coxeter number $h$). See the appendix for a proof that three variants of the Coxeter number, including $g$, are all integers.

\subsection{The finite subgroups of $\HH^\times$ and their McKay graphs} Let $\Gamma$ be a finite subgroup of $\HH^\times$. Identifying $\HH$ with $\CC^2$, the group $\Gamma$ is identified with a finite subgroup of $\mathrm{SL}_2(\CC)$. Its \emph{McKay graph} is the graph whose vertex set $I$ is the set of isomorphism classes of irreducible $\CC \Gamma$-modules, with $d$ edges connecting $i$ and $j$ if $j$ appears with multiplicity $d$ in $i \otimes \CC^2$ (equivalently, if $i$ appears with multiplicity $d$ in $j \otimes \CC^2$). McKay (see for example \cite{McK1} and the announcement \cite{McK2}) observed that the graphs arising in this fashion are precisely the affine Dynkin diagrams of type $ADE$: the cyclic subgroups produce type $A$, the binary dihedral groups produce type $D$, and the binary platonic groups produce type $E$. Writing $K(\CC \Gamma)$ for the Grothendieck group of the category of finite-dimensional $\CC \Gamma$ modules, we then have 
$$K(\CC \Gamma) \cong \ZZ^I,$$ identifying $K(\CC \Gamma)$ with the root lattice $\ZZ^I$ associated with the Dynkin diagram of $\Gamma$. Thus given a $\CC \Gamma$-module $M$, its \emph{character} is the class $\mathrm{ch}(M) \in \ZZ_{\geq 0}^I \subseteq \ZZ^I$ of $M$ in the Grothendieck group, which we may also visualize as a labeling of the Dynkin diagram by non-negative integers. In particular, the character of the regular representation of $\Gamma$ is the fundamental imaginary root $\delta$, that is $\mathrm{ch}(\CC \Gamma)=\delta$.

\subsection{The root system} \label{root system} In accordance with our identification of $\ZZ^I$ with the root lattice, we will sometimes write $\alpha_i$ for the simple root associated with the $i$th vertex of $I$, which has coefficient $1$ on $i$ and $0$ on each $j \neq i$. With this convention,
$$\delta=\sum_{i \in I} n_i \alpha_i,$$ where $n_i=\mathrm{dim}_\CC(i)$. The root system associated with the Dynkin diagram may then be described as follows: the roots are those elements of the form $n \delta+\alpha$, where $n \in \ZZ$ and $\alpha$ is a root of the finite root system associated with the Dynkin diagram obtained by deleting the vertex $0$ corresponding to the trivial representation of $\Gamma$. Thus we will write $\alpha_0$ for the simple root corresponding to the vertex of the trivial representation of $\Gamma$. The positive roots are those of the form $n \delta +\alpha$, where $n \geq 0$, $\alpha$ is a finite root, and $\alpha$ is positive if $n=0$. 

\subsection{The groups $W_n(\Gamma,\Delta)$} We now fix a positive integer $n$, a finite subgroup $\Gamma$ of $\HH^\times$, and a normal subgroup $\Delta \leq \Gamma$ such that $\Gamma / \Delta$ is abelian. We will define a certain subgroup $W_n(\Gamma,\Delta) \leq \mathrm{U}_n(\HH)$ associated with these data. Given $\gamma \in \Gamma$ and an integer $1 \leq p \leq n$, we write $\gamma^{(p)}$ for the diagonal matrix with $\gamma$ in position $(p,p)$ and with $1$ in every other diagonal position. We identify the symmetric group $S_n$ with the subgroup of $\mathrm{U}_n(\HH)$ consisting of permutation matrices (those with exactly one $1$ in each row and each column, and $0$'s in the other positions). 

Next we put
$$W_n(\Gamma)=\{\gamma_1^{(1)} \gamma_2^{(2)} \cdots \gamma_n^{(n)} w \ | \ \gamma_1,\gamma_2,\dots,\gamma_n \in \Gamma, \ w \in S_n \} \leq \mathrm{U}_n(\HH) .$$ Thus for $n=1$ we have $W_1(\Gamma)=\Gamma$, and for $\Gamma=\{1 \}$ we have $W_n(\Gamma)=S_n$. In general, $W_n(\Gamma)$ consists of matrices with exactly one non-zero entry in each row and each column, and so that the non-zero entries are elements of $\Gamma$. Finally we define
$$W_n(\Gamma,\Delta)=\{\gamma_1^{(1)} \gamma_2^{(2)} \cdots \gamma_n^{(n)} w \in W_n(\Gamma) \ | \ \gamma_1 \gamma_2 \cdots \gamma_n \in \Delta \}.$$ We note that in the definition of $W_n(\Gamma,\Delta)$, the order of the factors in the product $\gamma_1 \gamma_2 \cdots \gamma_n$ is irrelevant since $\Gamma/\Delta$ is abelian. Thus $W_n(\Gamma,\Delta)$ consists of $n$ by $n$ quaternionic matrices with exactly one non-zero entry in each row and each column, and so that the non-zero entries are elements of $\Gamma$ with product in $\Delta$. When $n=1$ we have $W_1(\Gamma,\Delta)=W_1(\Delta)=\Delta$ so we only obtain a new group for $n>1$. 

\subsection{Reflections and some numerology for $W_n(\Gamma,\Delta)$} Each group $W_n(\Gamma,\Delta)$ is generated by the set $R$ of reflections it contains. For $W_n(\Gamma)$, this set $R$ consists of two sorts of elements:
\begin{itemize}
\item[(a)] For each $\gamma \in \Gamma$ and each pair $1 \leq p < q \leq n$ of integers, the element
$$\gamma^{(p)} (pq) (\gamma^{(p)})^{-1},$$ where $(pq)$ is the transposition matrix interchanging the $p$th and $q$th canonical basis vectors of $\HH^n$ and leaving the others fixed, and

\item[(b)] for each $\gamma \in \Gamma \setminus \{ 1\}$ and each integer $1 \leq p \leq n$, the element $\gamma^{(p)}$.

\end{itemize} Therefore $W_n(\Gamma)$ contains 
$$N={n \choose 2} |\Gamma|+n(|\Gamma|-1)$$ reflections. Of these, the reflections that belong to $W_n(\Gamma,\Delta)$ are all those of type (a), together with the reflections of type (b) such that $\gamma \in \Delta$. Hence $W_n(\Gamma,\Delta)$ contains 
\begin{equation} \label{N for W} N={n \choose 2} |\Gamma|+n(|\Delta|-1) \end{equation} reflections, and in particular we have
\begin{equation} \label{g for W} g=2N/n=(n-1)|\Gamma|+2(|\Delta|-1). \end{equation}

\subsection{Semi-direct product algebra} Let $F$ be a field and let $A$ be an $F$-algebra equipped with an action by automorphisms of a group $G$. The \emph{semi-direct product algebra} (or \emph{twisted group ring} or \emph{smash product algebra}) $A \rtimes G$ is 
$$A \otimes_F FG \quad \text{with multiplication given by} \quad (a_1 \otimes g_1)(a_2 \otimes g_2)=a_1 g_1(a_2) \otimes g_1 g_2$$  for $a_1,a_2 \in A$ and $g_1,g_2 \in G$. For instance, given a field $F$, an $F$-vector space $V$, and a linear group $W \leq \mathrm{GL}(V)$, the group $W$ acts by automorphisms on the tensor algebra $T(V)$ of $V$, as well as on various related algebras such as the algebra $F[V]$, so we may form the associated semi-direct product algebras.

\subsection{Deformations} \label{deformations} In this subsection and the next two, we will work with an arbitrary field $F$, a finite-dimensional $F$-vector space $V$, and a finite subgroup $W \leq \mathrm{GL}(V)$. Our applications will all involve $F=\CC$, but here we present the simple facts we will need in their natural generality.

Let $A$ be a (unital, associative) $F$-algebra. A (non-negative, increasing, exhaustive) \emph{filtration} on $A$ is a sequence $(A^{\leq d})_{d \in \ZZ_{\geq 0}}$ of subspaces of $A$ indexed by non-negative integers $d \in \ZZ_{\geq 0}$ such that
$$1 \in A^{\leq 0}, \quad A^{\leq d} \subseteq A^{\leq e} \quad \hbox{if $d \leq e$,} \quad A^{\leq d} A^{\leq e} \subseteq A^{\leq d+e} \quad \text{and} \quad A=\bigcup_{d=0}^\infty A^{\leq d}.$$ The \emph{associated graded algebra} of $A$ with respect to this filtration is
$$\mathrm{gr}(A)=\bigoplus_{d=0}^\infty A^{\leq d} / A^{\leq d-1},$$ where $A^{\leq -1}=0$. The multiplication is induced by the multiplication of $A$. 

A \emph{deformation} of $F[V] \rtimes W$ is a filtered $F$-algebra $A$ equipped with an isomorphism of graded $F$-algebras $\phi: F[V] \rtimes W \to \mathrm{gr}(A)$. We note that upon restricting to the degree $0$ subalgebras, the map $\phi$ induces an isomorphism from $F W$ to $A^{\leq 0}$, so in particular $F W$ may be identified with the subalgebra of degree $0$ elements of $A$. Assuming that the order of $W$ is invertible in $F$, the \emph{spherical subalgebra} of a deformation $A$ of $F[V] \rtimes W$ is the idempotent slice subalgebra $eAe$, where 
$$e=\frac{1}{|W|} \sum_{w \in W} w$$ is the symmetrizing idempotent. If we think of the inclusion $F [V]^W \cong e(F[V] \rtimes W) e \subseteq F[V] \rtimes W$ as being a non-commutative version of the quotient map $V \to V/W$, then the inclusion $eAe \subseteq A$ should likewise be regarded as a non-commutative version of this quotient map. Taking this philosophy seriously leads to the definitions in the next paragraph.

\subsection{Non-commutative points and fibers} \label{non-commutative fibers}

Assuming as above that the order $|W|$ of $W$ is invertible in $F$, a \emph{non-commutative point} of $V/W$ is a one-dimensional representation of $eAe$ for some deformation $A$ of $F[V] \rtimes W$. A \emph{non-commutative fiber} of $V \to V/W$ is a representation $M$ of some deformation $A$ of $F[V] \rtimes W$ with the properties
\begin{itemize}
\item[(a)] $\mathrm{dim}_F(M^W)=1$, and
\item[(b)] $M$ is generated by $M^W$ as an $A$-module.
\end{itemize} Given a non-commutative fiber $M$ we define a filtration $M^{\leq d}$ on $M$ by 
$$M^{\leq d}=A^{\leq d} M^W.$$ Each non-commutative point $N$ produces a non-commutative fiber $M=Ae \otimes_{eAe} N$, and conversely each non-commutative fiber $M$ produces a non-commutative point $N=eM \cong M^W$. It seems likely that the classification problem for non-commutative points and fibers will be very interesting and challenging (for instance, by analogy with the same problem for finite $W$-algebras; see e.g. \cite{Los}).

The next lemma explains the utility of non-commutative fibers in the context of the zero-fiber ring. It is a version of Lemma 3.1 from \cite{Gri4}.

\begin{lemma} \label{fiber lemma}
Let $M$ be a non-commutative fiber.  Then $F[V]^W_{>0}$ acts by zero on $\mathrm{gr}(M)$, which is therefore a quotient of $F[\pi^{-1}(0)]$. In particular, the degree of the zero fiber is at least the dimension of $M$.
\end{lemma} 
\begin{proof}
First we observe that, since $M^W$ generates $M$ as an $A$-module, the filtration $M^{\leq d}$ is exhaustive and hence the dimension of $\mathrm{gr}(M)$ is equal to the dimension of $M$. By construction, $\mathrm{gr}(M)$ is generated as a $F[V]$-module by its degree $0$ subspace: given $\overline{m} \in \mathrm{gr}_d(M)=M^{\leq d} / M^{\leq d-1}$ there are $a \in A^{\leq d}$ and $m_0 \in M^W$ with $m=a m_0$, implying $$\overline{m}=\overline{a} m_0 \in (F[V] \rtimes W) \mathrm{gr}_0(M)=F[V] \mathrm{gr}_0(M).$$ Next, observe that the unique occurrence of the trivial representation of $W$ in $\mathrm{gr}(M)$ lies in degree $0$. This implies that each positive degree $W$-invariant polynomial acts by $0$ on the degree $0$ part of $\mathrm{gr}(M)$ and hence on all of $\mathrm{gr}(M)$. Hence $\mathrm{gr}(M)$ is a quotient of the zero-fiber ring $F[\pi^{-1}(0)]$.
\end{proof} 

\subsection{Normal subgroups} Next suppose $W \lhd N \leq \mathrm{GL}(V)$ is a pair consisting of two finite linear groups with $W$ normal in $N$; we have in mind here the pair $W_n(\Gamma,\Delta) \lhd W_n(\Gamma)$. The group $N$ then stabilizes the set of positive-degree $W$-invariants in $F[V]$ and hence acts by automorphisms on the zero-fiber ring $F[\pi^{-1}(0)]$ for $W$. We have the following analog of the previous lemma.
\begin{lemma} \label{normal subgroup fiber lemma}
Let $A$ be a deformation of $F[V] \rtimes N$ and let $M$ be an $A$-module with the properties
\begin{itemize}
\item[(a)] $\mathrm{dim}_F(M^W)=1$, and
\item[(b)] $M$ is generated by $M^W$ as an $A$-module. 
\end{itemize}  Then $F[V]^W_{>0}$ acts by zero on $\mathrm{gr}(M)$, which is therefore a quotient of $F[\pi^{-1}(0)]$. In particular, the degree of the zero fiber is at least the dimension of $M$.
\end{lemma} 
\begin{proof}
The proof is the same, once we observe that
$$(F[V] \rtimes N) \mathrm{gr}_0(M)=F[V] \mathrm{gr}_0(M),$$ since $N$ acts by some linear character on $\mathrm{gr}_0(M)=M^W$. 
\end{proof}

\subsection{The symplectic reflection algebra of rank $1$} Here we define the deformation of $\CC[x,y] \rtimes \Gamma$ introduced in \cite{CBH}, also known as the \emph{symplectic reflection algebra} (as introduced in more generality in \cite{EtGi}) of $\Gamma$. Thus let $\Gamma \subseteq \mathrm{SL}_2(\CC)$ be a finite subgroup, and choose linear functions $x,y$ on $V=\CC^2$ with $\la x,y \ra=1$, where $\la \cdot,\cdot \ra$ is the symplectic form on $V^*$ that is dual to the standard symplectic form on $V$. We have $\CC[V] = \CC[x,y]$. We write $I$ for an index set for the irreducible $\CC \Gamma$-modules (thus as above $I$ is the set of vertices of the Dynkin diagram of $\Gamma$), and for $i \in I$ let $e_i \in Z(\CC \Gamma)$ be the corresponding central primitive idempotent. We let $c=(c_i)_{i \in I} \in \CC^I$ be a collection of complex numbers indexed by $I$ and put
$$H_c(\Gamma,V)=\CC \la x,y \ra \rtimes \Gamma / ( xy-yx=\sum_{i \in I} c_i e_i ),$$ where $\CC \la x,y \ra$ is the tensor algebra of $V^*$, on which $\Gamma$ acts by automorphisms. The \emph{PBW theorem} (see e.g. \cite{Gri1}, Theorem 2.1) for $H_c(\Gamma,V)$ states that the monomials $x^a y^b \gamma$ for $a, b \in \ZZ_{\geq 0}$ and $\gamma \in \Gamma$ are a $\CC$-basis of $H_c(\Gamma,V)$. Equipping $H_c(\Gamma,V)$ with the filtration for which $x$ and $y$ have degree $1$ and elements of $\Gamma$ have degree $0$, this implies that $H_c(\Gamma,V)$ is a deformation of $\CC[x,y] \rtimes \Gamma$. 

\subsection{A deformation of $\CC[V^{\times n}] \rtimes W_n(\Gamma)$} We have $V^{\times n}=\HH^n$, which we regard as a $\CC$-vector space on which $W_n(\Gamma)$ acts $\CC$-linearly by restriction of scalars. Fixing a parameter $c$ as above, we form the tensor product algebra $H_c(\Gamma,V)^{\otimes n}$. This algebra is equipped with the permutation action of $S_n$ on the tensor factors, and we form the semi-direct product algebra
$$H_c(\Gamma,V)^{\otimes n} \rtimes S_n.$$ The tensor product filtration on $H_c(\Gamma,V)^{\otimes n}$ extends to a filtration on $H_c(\Gamma,V)^{\otimes n} \rtimes S_n$ by placing $S_n$ in degree $0$, and with this filtration the algebra $H_c(\Gamma,V)^{\otimes n} \rtimes S_n$ is a deformation of $\CC[V^{\times n}] \rtimes W_n(\Gamma)$. It is a special case of the symplectic reflection algebra of type $W_n(\Gamma)$ (see definition 2.1 and the beginning of  section 3 of \cite{EtMo}). 

\subsection{Representations} Given an $H_c(\Gamma,V)$-module $M$ and a $\CC S_n$-module $N$, the vector space $M^{\otimes n} \otimes N$ is a $H_c(\Gamma,V)^{\otimes n} \rtimes S_n$-module.
\begin{lemma} \label{irreducibility lemma} If $M$ and $N$ are irreducible finite-dimensional modules for $H_c(\Gamma,V)$ and $\CC S_n$, respectively, then  $M^{\otimes n} \otimes N$ is an irreducible finite-dimensional $H_c(\Gamma,V)^{\otimes n} \rtimes S_n$-module.
\end{lemma} 
\begin{proof}
The proof of this is very nearly standard, but we indicate it briefly here for the reader's convenience. It suffices to show that the canonical map $$H_c(\Gamma,V)^{\otimes n} \rtimes S_n \longrightarrow \mathrm{End}_\CC(M^{\otimes n} \otimes N)$$ is surjective. But this follows from the surjectivity of the maps $H_c(\Gamma,V) \to \mathrm{End}_\CC(M)$, $\CC S_n \to  \mathrm{End}_\CC(N)$, and $\mathrm{End}_\CC(M)^{\otimes n} \otimes  \mathrm{End}_\CC(N) \to  \mathrm{End}_\CC(M^{\otimes n} \otimes N)$, upon noting that given $w \in S_n$ its action on $M^{\otimes n}$ may be achieved by some element of $H_c(\Gamma,V)^{\otimes n}$. 
\end{proof}

\subsection{Finite-dimensional representations of $H_c(\Gamma,V)$.} Let $M$ be a finite-dimensional $H_c(\Gamma,V)$-module. As above, its \emph{character} $\mathrm{ch}(M)$ the element of the Grothendieck ring $K_0(\CC \Gamma) \cong \ZZ^I$ represented by its class as a $\CC \Gamma$-module. 

Writing $\alpha \cdot c=\sum_{i \in I} c_i d_i$ if $\alpha=\sum_{i \in I} d_i \alpha_i$ for the usual dot product and with the notation above for the root system $R$ associated with the Dynkin diagram of $\Gamma$, we define
$$R_c=\{ \alpha \in R \ | \  \alpha \cdot c =0 \}$$ and let $\Sigma_c$ be the set of minimal positive elements of $R_c$. Then Corollary 3.5 and Theorem 7.4 of \cite{CBH} imply:
\begin{theorem}
The correspondence $M \mapsto \mathrm{ch}(M)$ defines a bijection from the set of isomorphism classes of finite-dimensional simple $H_c(\Gamma,V)$-modules to the set $\Sigma_c$. 
\end{theorem}
As a corollary:
\begin{corollary} \label{finite dimensional corollary}
Let $\alpha$ be a positive real root. For generic elements $c$ of the hyperplane defined by $c \cdot \alpha=0$, there is a unique simple finite-dimensional $H_c(\Gamma,V)$-module $M$, and we have $\mathrm{ch}(M)=\alpha$. 
\end{corollary}

The corollary provides a large supply of non-commutative fibers. The largest character of such a non-commutative fiber is given by $\mathrm{ch}(M)=\delta+\phi$, where $\delta$ is the fundamental imaginary root and $\phi$ is the highest root of the finite root system associated with $\Gamma$. Recalling that $\delta=\mathrm{ch}(\CC \Gamma)$ and $\phi=\delta-\alpha_0$, the dimension of this $M$ is precisely 
$$\mathrm{dim}(M)=2|\Gamma|-1=g+1.$$ This shows that the dimension of the zero-fiber for $\Gamma$ is at least $g+1$. In the next section, we will observe that in fact this dimension is equal to $g+1$, and also generalize the preceding construction to the groups $W_n(\Gamma,\Delta)$. 

\section{Proof of the main results}

\subsection{Outline} Here we prove Theorem \ref{main}. We will begin with the case of a finite subgroup $\Gamma \leq \mathrm{SL}_2(\CC)$, proving that in this case the zero fiber ring is of dimension precisely $g+1$ and hence that the non-commutative fiber we constructed above is in fact a deformation of it. We then consider the groups $W_n(\Gamma,\Delta) \lhd W_n(\Gamma)$.

\subsection{Finite subgroups of $\HH^\times$} Let $V=\HH$, regarded as a right $\CC$-vector space via restriction of scalars, and let $\Gamma \subseteq \HH^\times$ be a finite subgroup acting on $V$ by left multiplication. Regarding $V$ as a $\CC$-vector space, the quotient space $V/\Gamma$ is known as a \emph{Kleinian singularity} (or alternatively, as a \emph{du Val singularity} or \emph{simple surface singularity}). The fibers of the quotient map are all reduced of cardinality $|\Gamma|$ except the zero-fiber. We will prove
\begin{theorem} \label{Kleinian case}
The degree of the zero-fiber of the map $V \to V/\Gamma$ is $2 |\Gamma| -1=g+1$. 
\end{theorem} The theorem is of course easy to deduce from the work of McKay \cite{McK3} on the graded character of the zero-fiber ring. We present a proof here for the reader's convenience; our interest in this particular formulation is the existence of a single irreducible non-commutative fiber with precisely the dimension of the zero-fiber ring. Already in this case, it would be very interesting to have an \emph{a priori} explanation for the coincidence of the dimensions. Additionally, the proof we give provides an explicit Groebner basis for the ideal generated by the positive degree $\Gamma$-invariants.

\begin{proof} The proof is case by case. In each case we will find a set $S$ of polynomials in the ideal $I$ generated by the $\Gamma$-invariant polynomials of positive degree such that the quotient of $\CC[x,y]$ by the ideal $\mathrm{in}(I)$ generated by the initial terms of the elements of $S$ is of dimension $2 |\Gamma|-1$. This gives the desired upper bound on the dimension of the zero-fiber ring, and when combined with the lower bound that follows from our previous construction of a non-commutative fiber of dimension $g+1$, establishes that the set $S$ is, in each case, a Groebner basis of $I$. Throughout we use the lexicographic order on monomials, so that $x^a y^b$ comes before $x^c y^d$ if $a>c$ or if $a=c$ and $b>d$. 

\subsection{Cyclic groups}

Let $\zeta$ be a primitive $\ell$th root of $1$ and let $\Gamma$ be the subgroup of $\mathrm{SL}_2(\CC)$ generated by the matrix
$$w=\left( \begin{matrix} \zeta & 0 \\ 0 & \zeta^{-1} \end{matrix} \right).$$ The action of $w$ on the ring $\CC[x,y]$ of polynomial functions on $\CC^2$ is given by
$$w \cdot x=\zeta^{-1} x \quad \text{and} \quad w \cdot y=\zeta y.$$

It follows that the polynomials $x^\ell,xy$ and $y^\ell$ are $\Gamma$-invariant. The quotient by the ideal they generated is spanned by the monomials $1,x,\dots,x^{\ell-1},y,y^2,\dots,y^{\ell-1}$ and hence has dimension at most $2 \ell -1=2|\Gamma|-1$.

\subsection{Binary dihedral groups $\mathrm{BD}_{2n}$ }

We write $\zeta$ for a primitive $2n$th root of $1$, and let
$$w_1=\left( \begin{matrix} \zeta & 0 \\ 0 & \zeta^{-1} \end{matrix} \right) \quad \text{and} \quad w_2=\left( \begin{matrix} 0 & i \\ i & 0 \end{matrix} \right).$$

We consider the polynomials
$$\phi_1=x^n+y^n, \quad \phi_2=x^n -y^n, \quad \text{and} \quad \phi_3=xy.$$ Direct calculation shows that these are semi-invariants with 
$$w_1 \cdot \phi_1=-\phi_1, \ w_2 \cdot \phi_1=(-i)^n \phi_1, \quad w_1 \cdot \phi_2=-\phi_2, \ w_2 \cdot \phi_2=-(-i)^n \phi_2, \quad w_1 \cdot \phi_3=\phi_3, \ \text{and} \ w_2 \cdot \phi_3=-\phi_3.$$

Hence if $n$ is even then the polynomials
$$f_1=\phi_3^2, \quad f_2=\phi_2^2, \quad \text{and} \quad f_3=\phi_1 \phi_2 \phi_3$$ are $\Gamma$-invariant, while if $n$ is odd then the polynomials 
$$g_1=\phi_3^2, \quad g_2=\phi_1 \phi_2, \quad \text{and} \quad g_3=\phi_2^2 \phi_3$$ are $\Gamma$-invariant.

We first assume $n$ is even, and in particular $n \geq 2$. We have
$$f_1=x^2y^2, \quad f_2=x^{2n}-2 x^n y^n+y^{2n}, \quad \text{and} \quad f_3=xy(x^{2n}-y^{2n}).$$ Hence the polynomial
$$y^{2n+2}=y^2 f_2-(x^{2n-2}-2 x^{n-2} y^n) f_1$$ belongs to $I$. Likewise
$$xy^{2n+1}=\frac{1}{2} \left(xy f_2-f_3+2x^{n-1} y^{n-1} f_1 \right)$$ belongs to $I$. The quotient by the ideal $J$ generated by the leading terms of the set
$$S=\{f_1,f_2,y^{2n+2},xy^{2n+1} \}$$ is spanned by
$$\{1,x,\dots,x^{2n-1},y,y^2,\dots,y^{2n+1}, xy,x^2y,\dots,x^{2n-1}y,xy^2,\dots,xy^{2n} \}$$ and is hence of dimension at most 
$$2n+(2n+1)+2n-1+2n-1=8n-1=2|\Gamma|-1.$$

Next we assume $n$ is odd. The invariant polynomials here are
$$g_1=x^2y^2, \quad g_2=x^{2n}-y^{2n}, \quad \text{and} \quad g_3=xy(x^{2n}-2x^ny^2+y^{2n}).$$ Since
$$y^{2n+2}=x^{2n-2} g_1-y^2g_2$$ we have $y^{2n+2} \in I$. And since
$$xy^{2n+1}=\frac{1}{2} \left(g_3-2x^{n-1} y^{n-1} g_1-xy g_2 \right)$$ we have $xy^{2n+1} \in I$. Thus again the desired dimension is bounded by $8n-1=2 |\Gamma|-1$. 

\subsection{The binary tetrahedral group} Here we follow \cite{Dol}, starting on page 8, and consider the binary tetrahedral group $\Gamma$ generated by the matrices
$$w_1=\left( \begin{matrix} i & 0 \\ 0 & -i \end{matrix} \right), \quad w_2=\left( \begin{matrix} 0 & i \\ i & 0 \end{matrix} \right), \quad \text{and} \quad 
w_3=\frac{1}{1-i}\left( \begin{matrix} 1 & i \\ 1 & -i \end{matrix} \right).$$ Direct calculation shows that
$$\phi_1=xy(x^4-y^4), \quad \phi_2=x^4+2 \sqrt{-3} x^2y^2+y^4, \quad \text{and} \quad \phi_3=x^4-2 \sqrt{-3} x^2y^2+y^4$$ are semi-invariants for $\Gamma$, and then that
$$f_1=\phi_1=xy(x^4-y^4), \quad f_2=\phi_2 \phi_3= x^8+14x^4y^4+y^8 \quad \text{and} \quad f_3=\frac{1}{2}(\phi_2^3+\phi_3^3)=x^{12}-33x^8y^4-33 x^4 y^8+y^{12}$$ are invariants for $\Gamma$. 

We define
$$g_1=yf_2-x^3f_1=15x^4y^5+y^9 \quad \text{and} \quad g_2=xg_1-15y^4f_1=xy^9,$$ so that $g_1,g_2 \in I$, and then put
$$h=f_3+(47y^4-x^4)f_2=624x^4y^8+48y^{12}.$$ Finally, defining
$$g_3=5h-208y^3g_1=32y^{12} \in I$$ we have obtained the desired set
$$S=\{f_1,f_2,g_1,g_2,g_3\}$$ whose initial terms $$\{x^5y,x^8,x^4y^5,xy^9,y^{12}\}$$ generate an ideal of codimension $47=2|\Gamma|-1$.

\subsection{The binary octahedral group} Here we follow \cite{Dol}, starting on page 9, and consider the binary octahedral group $\Gamma$ generated by the matrices
$$w_1=\left( \begin{matrix} e^{2 \pi i/8} & 0 \\ 0 & e^{-2\pi i /8} \end{matrix} \right), \quad w_2=\left( \begin{matrix} 0 & i \\ i & 0 \end{matrix} \right), \quad \text{and} \quad 
w_3=\frac{1}{1-i}\left( \begin{matrix} 1 & i \\ 1 & -i \end{matrix} \right).$$ Since the binary tetrahedral group is a normal subgroup of $\Gamma$, it is not surprising that its invariants
$$\phi_1=xy(x^4-y^4), \quad \phi_2=x^8+14x^4y^4+y^8, \quad \text{and} \quad \phi_3=x^{12}-33x^8y^4-33x^4y^8+y^{12}$$ are semi-invariants for $\Gamma$. Direct calculation then shows that
$$f_1=\phi_1^2=x^{10} y^2-2x^6y^6+x^2 y^{10}, \quad f_2=\phi_2=x^8+14x^4y^4+y^8, \quad \text{and} \quad f_3=\phi_1 \phi_3=x^{17}y-34x^{13}y^5+34x^5y^{13}-x y^{17}$$ are $\Gamma$ invariants. 

We define
$$g_1=\frac{1}{16}(x^2y^2 f_2-f_1)=x^6 y^6 \in I$$ and then
$$g_2=y^6f_2-x^2 g_2=14x^4 y^{10}+y^{14} \in I \quad \text{and} \quad g_3=x^2y^6(x^8+14x^4y^4+y^8)-(x^4+14y^4)g_2=x^2y^{14} \in I.$$ Next we observe that, working modulo $g_2=x^6 y^6$, we have
$$y^{10}(x^8+14x^4+y^8)=14 x^4 y^{14}+y^{18}$$ and, still working modulo $g_2=x^6 y^6$,
$$14 x^4y^{14}+y^{18}-14 x^2 y^4 f_1=y^{18} \in I.$$ Thus we define $g_4=y^{18} \in I$. Finally we observe that
$$g_5=(7x^9y-336x^5y^5+41 x y^9)f_2+4656x^3y^3 g_2-7 f_3=48xy^{17} \in I.$$ Hence the set
$$S=\{f_2,g_1,g_2,g_3,g_4,g_5\}$$ is contained in $I$, with initial terms the set
$$\{x^8,x^6y^6,x^4y^{10},x^2y^{14},xy^{17},y^{18} \}.$$ It follows that the degree of the zero fiber here is at most $95=2|\Gamma|-1$.

\subsection{The binary icosahedral group} Here we follow \cite{Dol}, starting on page 9. We take $\zeta=e^{2 \pi i /5}$ a primitive $5$th root of unity, and let $\Gamma$ to be the group generated by 
$$w_1=\left( \begin{matrix} e^{2 \pi i/10} & 0 \\ 0 & e^{-2\pi i /10} \end{matrix} \right), \quad w_2=\left( \begin{matrix} 0 & i \\ i & 0 \end{matrix} \right), \quad \text{and} \quad 
w_3=\frac{1}{\sqrt{5}}\left( \begin{matrix} \zeta-\zeta^4 & \zeta^2-\zeta^3 \\ \zeta^2-\zeta^3 & -\zeta+\zeta^4 \end{matrix} \right).$$ This group is perfect, so its semi-invariants are invariants, and as in \cite{Dol} we observe that the polynomials
$$f_1=xy(x^{10}+11x^5y^5-y^{10}), \quad f_2=-(x^{20}+y^{20})+228(x^{15}y^5-x^5y^{15})-494x^{10}y^{10},$$ and $$f_3=x^{30}+y^{30}+522(x^{25}y^5-x^5y^{25})-10005(x^{20}y^{10}+x^{10}y^{20})$$ are semi-invariants (and hence invariants). We put
$$g_1=(x^9-239x^4y^5)f_1+yf_2=-3124x^{10}y^{11}+11x^5y^{16}-y^{21} \in I$$ and
$$g_2=(x^{10}-239x^5y^5+3124y^{10})f_1+xyf_2=34375x^6y^{16}-3125xy^{21} \in I.$$ Next we set
$$g_3=\frac{1}{140}\left(3124^2(y^6f_2+x^9y^5f_1)+(3124 \cdot 239x^5-1543751y^5)g_1\right)=-16020500x^5y^{21}-58683y^{26}$$ and then
$$g_4=\frac{-1}{52081300000}(16020500y^5 g_2+34375xg_3)=xy^{26}.$$ Finally we set
$$h_1=f_3+x^{10}f_2-(750x^{14}y^4-18749x^9y^9)f_1=206761x^{15}y^{15}-28755x^{10}y^{20}-522x^5y^{25}+y^{30} \in I,$$ 
$$h_2=3124h_1+206761x^5y^4g_1=-87556200x^{10}y^{20}-1837490x^5y^{25}+3124y^{30} \in I,$$ and observe that modulo $g_2$ and $g_3$, the polynomial $h$ is a non-zero multiple of $g_5=y^{30}$, which is therefore in $I$. Thus the set
$$S=\{f_1,f_2,g_1,g_2,g_3,g_4,g_5\}$$ is a subset of $I$ with leading terms
$$\{x^{11}y,x^{20},x^{10}y^{11},x^6y^{16},x^5y^{21},xy^{26},y^{30} \}.$$ It follows as above that the dimension of the zero fiber is at most $239=2|\Gamma|-1$. This completes the proof of Theorem \ref{Kleinian case}, and hence Theorem \ref{main} in case $n=1$. \end{proof}

\subsection{Spaces of semi-invariants} We begin with a very general lemma, whose proof follows easily from Clifford theory.
\begin{lemma}
Let $G$ be a finite group, let $N \lhd G$ be a normal subgroup such that $G/N$ is abelian, let $\chi:G \to \CC^\times$ be a linear character, and let $M$ be an irreducible $\CC G$-module such that the restriction of $M$ to $N$ contains a summand isomorphic to the restriction of $\chi$ to $N$. Then $M$ is one-dimensional, and hence given by a linear character $\psi:G \to \CC^\times$ with restriction to $N$ equal to $\chi$ restricted to $N$.
\end{lemma}
\begin{proof}
By Clifford theory, the restriction of $M$ to $N$ is a sum of conjugates of $\chi$ restricted to $N$. But since $\chi$ is a character of $G$, it is conjugation-invariant, so the restriction of $M$ to $N$ is a sum of copies of $\chi|_N$, implying that $N$ acts via $\chi$ on $M$. Hence $N$ acts trivially on $M \otimes \chi^{-1}$, which is therefore an irreducible $G/N$ module. Since $G/N$ is abelian, it follows that $M$ is one-dimensional.
\end{proof}

Now fix a pair $(\chi,\eta)$ consisting of a linear character $\chi$ of $\Gamma$ and a linear character $\eta$ of $S_n$. This pair determines a linear character $(\chi,\eta)$ of $W_n(\Gamma)$ by the rule
$$(\chi,\eta)(\gamma_1^{(1)} \cdots \gamma_n^{(n)} w)=\eta(w) \cdot \prod_{i=1}^n \chi(\gamma_i),$$ and hence upon restriction to $W_n(\Gamma,\Delta)$, a linear character of $W_n(\Gamma,\Delta)$. Different pairs $(\chi,\eta)$ define different characters of $W_n(\Gamma)$, and all linear characters of $W_n(\Gamma)$ are of this form. Two pairs $(\chi,\eta)$ and $(\chi',\eta')$ restrict to the same character of $W_n(\Gamma,\Delta)$ if and only if $\eta=\eta'$ and the restrictions of $\chi$ and $\chi'$ to $\Delta$ are equal. Moreover, if $n \geq 3$ then all characters of $W_n(\Gamma,\Delta)$ are of this form. The following lemma is the key input in order to apply Lemmas \ref{fiber lemma} and \ref{normal subgroup fiber lemma}. In what follows, for a finite group $G$, an isoclass $\mu$ of simple representations of $\CC G$, and a $\CC G$-module $M$, we write $M^\mu$ for the sum of all submodules of $M$ of type $\mu$, and we write $G^\vee$ for the group of linear characters of $G$.
\begin{lemma} \label{semi-invariants lemma}
With notation as above, given a $\CC \Gamma$-module $L$, the space of $(\chi,\eta)$-relative invariants for the action of $W_n(\Gamma,\Delta)$ on $L^{\otimes n} \otimes \mathrm{det}$ is
$$(L^{\otimes n} \otimes \mathrm{det})^{(\chi,\eta)}=\begin{cases} \bigoplus\limits_{\substack{\psi \in \Gamma^\vee \\ \psi|_\Delta=\chi}} \Lambda^n (L^\psi) \quad \hbox{if $\eta=\mathrm{triv}$, and} \vspace{.05 in} \\  \bigoplus\limits_{\substack{\psi \in \Gamma^\vee \\ \psi|_\Delta=\chi}} \mathrm{Sym}^n (L^\psi) \quad \hbox{if $\eta=\mathrm{det}$.}\end{cases}$$
\end{lemma}

\begin{proof}
We put $$\Gamma^{\times n}_\Delta=\{(\gamma_1,\dots,\gamma_n) \in \Gamma^{\times n} \ | \ \gamma_1 \gamma_2 \cdots \gamma_n=1 \ \mathrm{mod} \ \Delta \},$$ which is a normal subgroup of $\Gamma^{\times n}$. For a representation $L$ of $\Gamma$ we compute the $W_n(\Gamma,\Delta)$ $(\chi,\eta)$-semi-invariants in $L^{\otimes n} \otimes \mathrm{det}$ as follows:  first we compute the $\Gamma^{\times n}_\Delta$ $\chi$-semi-invariants in $L^{\otimes n}$, and then take $S_n$ invariants or anti-invariants as appropriate. Suppose that $$L=\bigoplus_{j=1}^m L_j$$ with each $L_j$ an irreducible $\Gamma$-submodule of $L$. Now as a $G=\Gamma^{\times n}$-module, $L^{\otimes n}$ decomposes as 
$$L=\bigoplus_{1 \leq i_1,i_2,\dots,i_n \leq m} L_{i_1} \otimes L_{i_2} \otimes \cdots \otimes L_{i_n}.$$ By applying the previous lemma with $N=\Gamma^{\times n}_\Delta$, $M=L_{i_1} \otimes \cdots \otimes L_{i_n}$, and the character of $G$ induced by $\chi$, the only summands that contribute to the semi-invariants we are after must have all $L_{i_j}$ one-dimensional. Let $\chi_j$ be the linear character corresponding to $L_{i_j}$. Assuming that $\ell_1 \otimes \cdots \otimes \ell_n \in M$ is a $\chi$-semi-invariant we obtain
$$\ell_1 \otimes \cdots \ell_n=\gamma^{(1)} (\gamma^{-1})^{(2)} \cdot \ell_1 \otimes \ell_2 \otimes \cdots \ell_n=\chi_1(\gamma) \chi_2(\gamma^{-1}) \ell_1 \otimes \ell_2 \otimes \cdots \ell_n \quad \hbox{for all $\gamma \in \Gamma$,}$$ implying $\chi_1=\chi_2$. Similarly $\chi_i=\chi_j$ for all $i,j$. Likewise, for $\gamma \in \Delta$ we have
$$\chi(\gamma) \ell_1 \otimes \cdots \otimes \ell_n=\gamma^{(1)} \cdot \ell_1 \otimes \cdots \otimes \ell_n=\chi_1(\gamma) \ell_1 \otimes \cdots \otimes \ell_n,$$ implying that $\chi$ and $\chi_1$ have the same restriction to $\Delta$. Finally, by taking $S_n$-invariants or anti-invariants, as the case may be, we obtain the result.
\end{proof}

\subsection{Proof of Theorem \ref{main}} We have already proved the last assertions of the theorem, dealing with the rank one groups. The remainder of the theorem is a consequence of the following more precise version combined with Cohen's \cite{Coh} classification.
\begin{theorem}
Let $n \geq 2$ be an integer, let $\Gamma \leq \HH^\times$ be a non-trivial finite subgroup, and let $\Delta \leq \Gamma$ be a normal subgroup such that $\Gamma/\Delta$ is abelian. Then the zero-fiber ring of $W_n(\Gamma,\Delta)$ admits a $(g+1)^n$-dimensional quotient ring, where
$$g=(n-1)|\Gamma|+2(|\Delta|-1).$$
\end{theorem} 
\begin{proof}
We begin with the case $W_n(\Gamma)$, in order to illustrate the ideas without the case-by-case considerations that will be necessary for the groups $W_n(\Gamma,\Delta)$. Using Corollary \ref{finite dimensional corollary} we choose a parameter $c \in \CC^I$ such that there is an irreducible $H_c(\Gamma,V)$-module $L$ with character $\mathrm{ch}(L)=n\delta+\phi$. By Lemma \ref{irreducibility lemma}, the representation $L^{\otimes n} \otimes \mathrm{det}$ of $H_c(\Gamma,V)^{\otimes n} \rtimes S_n$ is irreducible. By Lemma \ref{semi-invariants lemma}, the space of $W_n(\Gamma)$-invariants in this representation is one-dimensional, so it is a non-commutative fiber. By Lemma \ref{fiber lemma}, its associated graded module is a quotient ring of the zero-fiber ring, and since its dimension is
$$\mathrm{dim}_\CC \left(L^{\otimes n} \otimes \mathrm{det}\right)=\mathrm{dim}(L)^n=\left(n |\Gamma|+|\Gamma|-1 \right)^n=(g+1)^n,$$ we have completed the proof in this case. 

Now we handle the case of $W_n(\Gamma,\Delta)$ for $\Delta \neq \Gamma$. We may assume $\Gamma$ is not abelian, since in that case the results of \cite{Gri4} apply (except for the group $W_2(\Gamma,\{1\})=G(2,2,2)$, which is not irreducible, but one can check the theorem directly in that case). By arguing as above with Lemma \ref{irreducibility lemma}, Lemma \ref{semi-invariants lemma}, and Lemma \ref{normal subgroup fiber lemma}, we search for an irreducible $H_c(\Gamma,\CC^2)$-module $L$ with character as large as possible subject to the condition that 
\begin{equation} \label{dimension bound}
\mathrm{dim}(L^\chi) \leq n \quad \hbox{for all $\chi \in (\Gamma/\Delta)^\vee$, with equality for exactly one such $\chi \in (\Gamma/ \Delta)^\vee$.}
\end{equation}   Given a positive root of the form $(n-1) \delta+\alpha$ for a finite positive root $\alpha=\sum k_i \alpha_i$ with the property that $k_i=1$ for exactly one index $i$ corresponding to a linear character of $\Gamma / \Delta$, Corollary \ref{finite dimensional corollary} implies that there is an $L$ with $\mathrm{ch}(L)=(n-1)\delta+\alpha$. (We could also take $L$ with $\mathrm{ch}(L)=n \delta-\beta$ for a positive root $\beta$ such that $\beta_i=1$ for all $i$ corresponding to non-trivial characters of $\Gamma/\Delta$). For this $L$, its dimension is
$$\mathrm{dim}(L)=\sum_{i \neq 0} k_i n_i+(n-1)|\Gamma|.$$ Recalling from \eqref{g for W} that 
$$g=2N/n=(n-1)|\Gamma|+2(|\Delta|-1),$$ the equality $\mathrm{dim}(L)=g+1$ is equivalent to
\begin{equation} \label{alpha condition} \sum_{i \neq 0} k_i n_i=2 |\Delta| -1. \end{equation} It remains to show that in each case we can choose $\alpha$ a finite root such that \eqref{alpha condition} holds, and so that there is precisely one character $\chi$ of $\Gamma/\Delta$, corresponding to a vertex $i \in I$, such that $k_i=1$, implying \eqref{dimension bound} for $L$ with $\mathrm{ch}(L)=(n-1)\phi + \alpha$. (As it turns out, if $\alpha$ is a maximal positive root with this latter property, then we will observe that it satisfies the former). 

The simplest case is that in which $|\Gamma / \Delta |=2$. In this case the largest possible choice $\alpha=\phi$ works, so that $\mathrm{ch}(L)=(n-1) \delta+\phi$. The remaining cases are the commutator subgroup $\Delta$ of index $3$ in the binary tetrahedral group (of type $E_6^{(1)}$), and the commutator subgroup (which is of order $n$, or equivalenty index $4$) of the binary dihedral group of order $4n$. We now complete the proof of Theorem \ref{main} by specifying, in each case, a particular $\alpha$ with the property \eqref{alpha condition} and such that $k_i=1$ for precisely one vertex $i$ corresponding to a non-trivial character of $\Gamma/\Delta$. For the reader's convenience we also label the Dynkin diagram with the integers $n_i$ giving the coefficients of $\delta$ on the simple roots.

\let\dlabel=\mlabel

$$
\delta \ \text{for type} \ E_6^{(1)} \quad
\begin{tikzpicture}
\begin{scope}[start chain]
\foreach \dyi in {1,2,3,2,1} {
\dnode{\dyi}
}
\end{scope}
\begin{scope}[start chain=br going above]
\chainin(chain-3);
\dnodebr{2}
\dnodebr{1}
\end{scope}
\end{tikzpicture}
\qquad
\text{and} \ \alpha \ \text{is} \quad
\begin{tikzpicture}
\begin{scope}[start chain]
\foreach \dyi in {1,2,2,1,0} {
\dnode{\dyi}
}
\end{scope}
\begin{scope}[start chain=br going above]
\chainin(chain-3);
\dnodebr{1}
\dnodebr{0}
\end{scope}
\end{tikzpicture}
$$
\vspace{.5 in}
$$
\delta \ \text{for type} \ D_{n+2}^{(1)} \quad 
\begin{tikzpicture}
\begin{scope}[start chain]
\dnode{1}
\dnode{2}
\dnode{2}
\dydots
\dnode{2}
\dnode{1}
\end{scope}
\begin{scope}[start chain=br going above]
\chainin(chain-2);
\dnodebr{1};
\end{scope}
\begin{scope}[start chain=br2 going above]
\chainin(chain-5);
\dnodebr{1};
\end{scope}
\end{tikzpicture}
\qquad 
\text{and} \ \alpha \ \text{is} \quad
\begin{tikzpicture}
\begin{scope}[start chain]
\dnode{1}
\dnode{1}
\dnode{1}
\dydots
\dnode{1}
\dnode{0}
\end{scope}
\begin{scope}[start chain=br going above]
\chainin(chain-2);
\dnodebr{0};
\end{scope}
\begin{scope}[start chain=br2 going above]
\chainin(chain-5);
\dnodebr{0};
\end{scope}
\end{tikzpicture} 
$$ \end{proof} 

We note that in the proof, we have searched for a maximum-dimensional representation to which Lemma \ref{normal subgroup fiber lemma} applies; it seems miraculous to us that such a representation turns out to be of dimension exactly $(g+1)^n$. 

\subsection{Catalan numbers for $W_n(\Gamma,\Delta)$?} \label{Catalan numbers} As mentioned in the introduction, it is natural to wonder if there are reasonable product formulas for the dimensions of the spaces of $(\chi,\mathrm{det})$ semi-invariants appearing in our non-commutative fibers, as happens for complex reflection groups. Even though the statement of Lemma \ref{semi-invariants lemma} seems \emph{a priori} discouraging, as it contains direct sums, one may observe by a case-by-case analysis that in fact the dimension always factors into a product of $n$ linear functions of $n$ divided by $n!$. We have collected the results of this straightforward calculation in the next table. We note that they very nearly do not depend on the pair $(\Gamma,\Delta)$ except through the index $|\Gamma/\Delta|$ (though presumably their graded versions do).

\[
\begin{array}{c|c}
\hline
\text{The pair} \ (\Gamma,\Delta) \ \text{and the conditions on} \ \chi & \text{The dimension of} \ (L^{\otimes n} \otimes \mathrm{det})^{(\chi,\mathrm{det})} \\
\hline
 & \\
\Gamma=\Delta, \quad \chi \neq 1 & \frac{2n(2n-1)(2n-2) \cdots (n+2)(n+1)}{n!} \\
 & \\
 \hline
 & \\
 \Gamma=\Delta, \quad \chi = 1 &  \frac{(2n-1)(2n-2) \cdots (n+1)n}{n!} \\
& \\
\hline
& \\
|\Gamma:\Delta|=2,  \quad \chi|_\Delta \neq 1 &  \frac{(4n-2)(2n-2)(2n-3) \cdots (n+1)n}{n!} \\
 & \\
 \hline
  & \\
 |\Gamma:\Delta|=2,  \quad \chi|_{\Delta}=1 &  \frac{(3n-2)(2n-2)(2n-3) \cdots (n+1)n}{n!} \\
 & \\ 
 \hline 
 & \\
 \Gamma \ \text{of type} \ E_6^{(1)}, \ |\Gamma:\Delta|=3 & \frac{(4n-3)(2n-2)(2n-3) \cdots (n+1)n}{n!} \\
 & \\
 \hline
 & \\ 
  \Gamma \ \text{of type} \ D_{m+2}^{(1)}, \ |\Gamma:\Delta|=4 & \frac{(5n-4)(2n-2)(2n-3)\cdots(n+1)n}{n!} \\
  & \\ 
  \hline
\end{array}
\]

\vspace{.1 in}

The results for $(\chi,\mathrm{triv})$ semi-invariants in the cases not covered above are less interesting, but we present them here:

\vspace{.1 in}

\[
\begin{array}{c|c}
\hline
\text{The pair} \ (\Gamma,\Delta) \ \text{and the conditions on} \ \chi & \text{The dimension of} \ (L^{\otimes n} \otimes \mathrm{det})^{(\chi,\mathrm{triv})} \\
\hline
 & \\
\Gamma=\Delta, \quad \chi \neq 1 & n+1 \\
 & \\
 \hline
& \\
|\Gamma:\Delta|=2,  \quad \chi|_\Delta \neq 1 &  2 \\
 & \\
 \hline
\end{array}
\]

\section{Appendix: some numerology for quaternionic groups}

\subsection{Variants of the Coxeter number} Let $V$ be an $n$-dimensional $\HH$-vector space equipped with a positive definite Hermitian form $(\cdot,\cdot)$. We assume given a finite group $W$ of unitary transformations of $\HH$ such that $V$ is irreducible (as an $\HH$-linear representation of $W$). Let 
$$R=\{r \in W \ | \ \mathrm{codim}(\mathrm{fix}(r))=1 \}$$ be the set of reflections in $W$. Note that we do \emph{not} require here that $R$ generates $W$ (or even that $R \neq \emptyset$).

 We write
$$\mathcal{A}=\{\mathrm{fix}(r) \ | \ r \in R \}$$ for the arrangement of reflecting hyperplanes for elements of $R$, and define $N=|R|$ and $N^*=|\mathcal{A}|$. Finally, we define
$$g=2N/n, \quad h=(N+N^*)/n, \quad \text{and} \quad k=2N^*/n.$$ Evidently we have $g+k=2h$ and $g \geq h \geq k$, with both equalities occurring precisely when every $r \in R$ has order $2$ (thus in particular when the $W$-module $V$ has an $\RR$-form). When $W$ is a complex reflection group, it is traditional to refer to $h$ as the \emph{Coxeter number} of $W$. In this appendix we will prove:
\begin{theorem}
\label{coxeternumbers}
The numbers $g$, $h$, and $k$ are all integers.
\end{theorem}

\subsection{Proof of Theorem \ref{coxeternumbers}}

We will prove the theorem after establishing a number of preparatory results. First, we prove that $h$ is an integer. 
\begin{lemma}
\label{h is an integer}
$h$ is an integer.
\end{lemma}
\begin{proof}
Let $z=\sum_{r\in R} 1-r \in \ZZ(\CC W)$. 
\textbf{Case 1:} In the case when $V\in \operatorname{Irr}\CC W$ is irreducible over $\CC$, it follows from Schur's lemma that $z$ acts on $V$ by scalar $c_V\in \CC$. Moreover, since $z$ is an integer-linear combination of class sums, $c_V$ is an algebraic integer.

By definition of $c_V$,
$$2nc_V=\operatorname{tr}_V(z)=\sum\limits_{r\in R}\operatorname{tr}_V(1-r)=\sum\limits_{H\in \mathcal{A}}\sum\limits_{r\in W_H}\operatorname{tr}_V(1-r)$$
where $W_H=\{r\in W | r(p)=p ,\, \forall p\in H\}$. Let $H^{\bot}$ be the orthogonal complement of $H$ in $V$, so that $V=H\oplus H^{\bot}$. Noting that, for $r \in W_H$, $1-r$ acts by zero on $H$ while $\sum_{r \in W_H} r$ acts by zero on$H^{\bot}$,
$$2nc_V=\sum\limits_{H\in \mathcal{A}}\sum\limits_{r\in W_H}\operatorname{tr}_{H^{\bot}}(1)=\sum\limits_{H\in \mathcal{A}}\sum\limits_{r\in W_H}2=\sum\limits_{H\in \mathcal{A}}2|W_H|=2(N+N^*) \, .$$
This implies that $h=c_V$ is an algebraic integer; since it is evidently a rational number it is an integer. 

\textbf{Case 2:} If $V$ is not irreducible as a $\CC W$-module, then it is of the form $V=V_0 \otimes_\CC \HH$ for some irreducible $\CC W$-module $V_0$, and the same argument applies. This shows that $h$ is an integer in either case.
\end{proof}

Next we prove that $k$ is an integer (which, together with the preceding shows that $g$ is an integer). This follows the proof of Corollary 6.98 of \cite{OrTe}, modified in order to apply to the quaternionic case and in order to avoid their appeal to the theorem that parabolic subgroups of reflection groups are reflection groups (this is known thanks to \cite{BST} for quaternionic groups, but via a case-by-case check). Fix $H\in \mathcal{A}$ and set $\mathcal{A}^H=\{ H\cap K |K\in \mathcal{A}\setminus \{H\} \}$. In order to prove that $k\in \ZZ$, we will in fact show that $|\mathcal{A}^H|=N^*+1-k$. For each $H \in \cA$, we fix $\alpha_H\in V$ such that 
\begin{enumerate}
    \item $H=\{v\in V |(\alpha_H, v)=0\}$,
    \item $(\alpha_H, \alpha_H)=1$.
\end{enumerate} Thus $\alpha_H$ is well-defined up to multiplication by a quaternion of norm $1$.
We will consider the function $f:V\rightarrow V$ defined by $v\mapsto \sum\limits_{H\in \mathcal{A}}\alpha_H(\alpha_H, v)$.

\begin{lemma} The function $f$ satisfies
    $f(v)=\frac{k}{2}v, \, \forall v\in V$.
\end{lemma}
\begin{proof}
\textbf{Case 1:} In the case when $V\in \operatorname{Irr}\CC W$ is irreducible over $\CC$. It follows from Schur's lemma that $f$ acts on $V$ by a scalar $c\in \CC$. By fixing an orthonormal $\HH$-basis $e_1,\dots,e_n$ of $V$ and computing $\sum_{i=1}^n (e_i,f(e_i))$ one finds that this scalar is $k/2$.

\textbf{Case 2:} In this case $V=V_0 \otimes_\CC \HH$ for an irreducible $\CC W$-module $V_0$, and the same argument applies upon working in $V_0$.
\end{proof}

\begin{corollary}
    For $H\in \mathcal{A}$, we have $$2\sum\limits_{K\in \mathcal{A}}|(\alpha_K, \alpha_H)|^2=k.$$  
\end{corollary}
\begin{proof} This follows from
    $$k/2 = (\alpha_H, f(\alpha_H))=\sum\limits_{K\in \mathcal{A}}(\alpha_H, \alpha_K(\alpha_K, \alpha_H)).$$
\end{proof}

\begin{lemma}
    For $X\in \mathcal{A}^H$, $|\{K\in \mathcal{A} \, | X\subseteq K \}| = 2\sum\limits_{\substack{X\subseteq K \\ K\in \mathcal{A}}}|(\alpha_k, \alpha_H)|^2$.
\end{lemma}
\begin{proof}
Let $$W_X=\{w \in W \ | \ w(x)=x \ \hbox{for all $x \in X$} \}$$ be the subgroup of $W$ that fixes $X$ pointwise. Thus $W_X$ acts on the orthogonal complement $X^\perp$, and the reflecting hyperplanes for $W_X$ acting on $X^\perp$ are precisely those intersections $K \cap X^\perp$ for which $K \in \cA$ satisfies $X \subseteq K$.

\textbf{Case 1:} We assume first that $W_X$ acts $\HH$-reducibly on $X^\perp$. Thus $X^\perp$ is the direct sum of two one-dimensional $W_X$-stable subspaces, and for each $K \in \cA$ with $K \supseteq X$ we must have $\alpha_K$ in one of these subspaces or the other. It follows that there are exactly two $K \in \cA$ with $X \subseteq K$ (one of which is $H$) and we have $(\alpha_H,\alpha_K)=0$. This establishes the lemma in this case. 

\textbf{Case 2:} Now assume $W_X$ acts $\HH$-irreducibly on $X^\perp$. In this case, we may apply the previous corollary, which proves the lemma in this case.
\end{proof}

\begin{proof} (of Theorem \ref{coxeternumbers})
By Lemma \ref{h is an integer}, $h\in \ZZ$. We next prove that $k$ is an integer. We may assume $\cA \neq \emptyset$; we fix $H \in \cA$.

By making use of the preceding results we calculate:
\begin{eqnarray*}
   k&=&2\sum\limits_{K\in \mathcal{A}}|(\alpha_K, \alpha_H)|^2=2 + 2\sum\limits_{K\in \mathcal{A}\setminus \{H\}}|(\alpha_K, \alpha_H)|^2= 2 + \sum\limits_{X\in \mathcal{A}^H}\bigl(-1 + \sum\limits_{\substack{X\subseteq K \\K\in \mathcal{A}}}|(\alpha_K, \alpha_H)|^2\bigr)\\
    &=&2-2|\mathcal{A}^H|+\sum\limits_{X\in \mathcal{A}^H}|\{K\in \mathcal{A} |\, X\subseteq K\}|=2-2|\mathcal{A}^H|+|\mathcal{A}^H| + \sum\limits_{X\in \mathcal{A}^H}|\{ K\in \mathcal{A} |\, K\cap H=X \}|
   \\
   &=&2-|\mathcal{A}^H|+|\mathcal{A}|-1=1+N^*-|\mathcal{A}^H|.
\end{eqnarray*}
\end{proof}
\bibliographystyle{a                      
                       msplain}
\def\cprime{$'$} \def\cprime{$'$}

\end{document}